%% file: main.tex
\documentclass[11 pt]{article}
\usepackage[margin=1in]{geometry}
\input{Compilation_Files/preamble}

\title{Counting cliques in graphs with small independence number}
\author{Logan Post\thanks{School of Mathematics, Georgia Institute of Technology, Atlanta, GA 30332. Email: lpost3@gatech.edu}}

\begin{document}

\maketitle

\abstract{We prove that for all fixed $k\geq 4$, any $N$ vertex graph with no independent set of size $n$ and $N\geq \Omega(n^{k-1}/\log^{k-2}n)$ contains at least
\[
\Omega\bigg(\binom Nk \Big(\frac{\log n}{n}\Big)^{\binom k2}/\log n\bigg)
\]
cliques of order $k$, and for $k\geq 5$ this is best possible conditional on the known upper bounds for $r(k,n)$. This is also true and tight for $k=2$ by Tur\'an's Theorem and for $k=3$ by a result of Bohman and Mubayi. We show the bound is also tight for $k=4$. We obtain other supersaturation results using the same methods.
}

\section{Introduction}
What is the minimum number of unlabeled $k$-cliques which must be contained in an $N$-vertex graph if it contains no independent set of size $n$? Erd\H{o}s \cite{Erdos1962} asked for this value, denoted $f(N,k,n)$, in $1962$, and a substantial line of work (see \cite{Das,Nikiforov,Pikhurko}) has studied the problem for fixed parameters $n$ and $k$. Here, we study the regime in which $n$ grows with $N$.

In the moving parameter regime, $f(N,k,n)$ generalizes the Tur\'an number $\rm{ex}(N,n)={\binom{N}{2}}-f(N,2,n)$. Additionally, $f$ captures the behavior of the off-diagonal Ramsey number $r(k,n)$, as we have $f(N,k,n)=0$ if and only if $N<r(k,n)$. The current best known bounds on the off-diagonal Ramsey numbers for fixed $k\geq 4$ are
\begin{equation}\label{eq:ramseybounds}
\Omega\Big(\frac{n^{k-1}}{\log^{2k-4}n}\Big)\leq r(k,n)\leq (1+o(1))\frac{n^{k-1}}{\log^{k-2}n}.
\end{equation}
The upper bound order is due to Ajtai, Koml\'os, and Szemer\'edi \cite{AKS} with the leading constant due to Shearer \cite{shearer} for $k=3$ and Li, Rousseau, and Zang \cite{LiRousseauZang2001} for $k\geq 4$. The lower bound is quite recent, due to Brada\v c \cite{bradac} (for $k=4$ see \cite{r4k}). These bounds estimate when $f=0$, and a routine averaging argument can be used to obtain `supersaturation' just over the threshold. A result of Bohman and Mubayi \cite{Bohman_Mubayi_2019} proves that when $N\geq \Omega(n^{k-1})$, we have a much larger abundance of $k$-cliques. When $k=3$, where it is known \cite{kim,r3k} that $r(3,n)=\Theta(n^2/\log n)$, they prove a much stronger existence-to-abundance phase transition at exactly $N=\Theta(r(3,n))$, and compute $f(N,3,n)$ up to a constant factor (in $N$).

In this paper, we obtain analogous lower bounds for all $k\geq 4$ down to a threshold which matches the upper bound in Equation \eqref{eq:ramseybounds} up to a constant. Conditional on the tightness of the upper bound in \eqref{eq:ramseybounds}, our results are also optimal up to constants; otherwise we are within one factor of $\log n$. Our result is a supersaturation theorem which illustrates a phase transition of $f$ at $N=\tilde \Theta(r(k,n))$ from $0$ to approximately random-like (see \eqref{eq:gdefn}). In addition, we determine the hierarchy of $i$-clique counts for $i<k$ up to a constant factor in graphs with small independent sets and without many $k$-cliques. We prove that our bounds are tight when $k=4$, and discuss the generalization $f(N,H,n)$ which counts copies of an arbitrary graph $H$.

\subsection{Main Results}

We begin with the theorem of Bohman and Mubayi from \cite{Bohman_Mubayi_2019} upon which we improve in this paper.

\begin{theorem}[Bohman--Mubayi]\label{thm:bohmanmubayi} Let $k\geq 2$ be a fixed constant. If $G$ is a graph on $N$ vertices with no independent set of size $n$ and containing $t$ copies of $K_k$, then 
\[
n> \begin{cases}
c_kN^{\frac{1}{k-1}}&\text{ if }t\leq N^{k/2}\\
c_k\big(\frac{N^k}{t}\big)^{\frac{1}{{\binom{k}{2}}}}&\text{ if }t\geq N^{k/2}\\
\end{cases}
\]
\end{theorem}
Up to the constant $c_k$, this is equivalent to the following statement using simple algebraic manipulations.

\begin{theorem}\label{thm:nologs} If $G$ is an $N$-vertex graph with no independent set of size $n$ and $N\geq 2n^{k-1}$, then $G$ contains at least $\frac 12{\binom{N}{k}}n^{-{\binom{k}{2}}}$ cliques of size $k$. Equivalently 
\[
f(N,k,n) \geq \frac 12{\binom{N}{k}}n^{-{\binom{k}{2}}}.
\]
\end{theorem}
The latter form is perhaps cleaner, more compatible with our proof techniques, and illustrates the tightness up to log factors with the random graph. Indeed, if $N=\poly(n)$ and $p=\Theta(\log n/n)$ then $\bG(N,p)$ has independence number $O(n)$ and contains on the order of
\begin{equation}\label{eq:gdefn}
g(N,k,n)\coloneqq 
{\binom{N}{k}}\Big(\frac{\log n}{n}\Big)^{{\binom{k}{2}}}
\end{equation}
cliques of size $k$. The random graph implies $f(N,k,n)\leq O_k(g(N,k,n))$ for $N=\poly(n)$, and \cref{lem:blowups} generalizes this bound to all $N\geq 1$. \cref{thm:nologs} leaves a large polylogarithmic gap. Our main theorem will sharpen the bounds on the transition point $\tilde\Theta(n^{k-1})$ and show that $f$ is actually much closer to $g$. Our result is new for $k\geq 4$.

\begin{theorem}\label{thm:main} For all $k\geq 2$, there exist $c_1,c_2>0$ such that if $G$ is an $N$-vertex graph with no independent set of size $n$ and $N\geq c_1 n^{k-1}/\log^{k-2}n$, then $G$ contains at least $c_2\cdot g(N,k,n)/\log n$ cliques of size $k$. Equivalently,
\[
f(N,k,n)\geq c_2{\binom{N}{k}}\Big(\frac{\log n}{n}\Big)^{\binom{k}{2}}/\log n.
\]
\end{theorem}

This is a quantitative supersaturation refinement of the bound of Ajtai, Koml\'os, and Szemer\'edi on $r(k,n)$, because in particular $f>0$ for $N=c_1n^{k-1}/\log^{k-2}n$. We are unable to recover the $(1+o(1))$ constant factor of Equation \eqref{eq:ramseybounds} to get a sharp threshold; our bound on $c_1$ is exponential in $k$.

When $k=2$, this is implied by Tur\'an's theorem. When $k=3$, this is equivalent to Theorem 2 in \cite{Bohman_Mubayi_2019} by similar algebraic manipulations. It is tight up to constants for $k=2,3,$ and we prove in Section \ref{sec:upperbounds} that it is tight when $k=4$.

This is an existence-to-abundance supersaturation result which significantly exceeds the bounds obtained by standard averaging. It is surprising that even though $\bG(N,p)$ is a poor $(k,n)$-Ramsey graph, once we grow even a constant multiple beyond the known Ramsey bounds, $\bG(N,p)$ becomes essentially the best possible at avoiding $K_k$. This suggests that specialized quasirandom constructions obtained by gluing efficient Ramsey graphs are not particularly effective at minimizing $k$-clique density.

With a natural additional condition, we can remove the remaining logarithmic factor. 

\begin{theorem}\label{thm:extralog}
For $k\geq 2$ and $\eps>0$, if $G$ is an $N$-vertex graph with no independent set of size $n$ and $N\geq n^{k-2+\eps}$, and additionally $G$ contains {at most} $n^{1-\eps}\cdot g(N,k,n)$ cliques of size $k$, then $G$ contains $\Omega_{k,\eps}(g(N,i,n))$ cliques of size $i$ for all $i<k$.
\end{theorem}

In particular, this theorem applies if $G$ is $K_{k}$-free, giving structural information about all efficient $(k,n)$ Ramsey graphs. 
The conclusion of \cref{thm:extralog} is clearly tight (depending on the constraints on $N$ and $m_{k}$) by taking $\bG(N,p)$; one might wonder if \cref{thm:main} could similarly be improved to remove the extra log. This is false for $k=2,3,4$, and using the 
construction of Bohman and Mubayi we show that depending on $r(k,n)$, the logarithmic loss may be genuine for all $k$.

\begin{theorem}\label{thm:tightness} Fix $k\geq 5$. If it holds that 
\[r(k,n)=\Theta_k(n^{k-1}/\log^{k-2}n),\]
then \cref{thm:main} is tight up to the choice of $c_1$ and $c_2$. In particular, there exist $a_1,a_2$ such that for all $N\leq a_1n^{k-1}/\log^{k-2}n$ we have $f(N,k,n)=0$, and for all $N\geq a_1n^{k-1}/\log^{k-2}n$ we have
\[
f(N,k,n)\leq a_2{\binom Nk}\Big(\frac{\log n}{n}\Big)^{\binom k2}/\log n.
\]
\end{theorem}

This theorem states that the tightness of Ajtai, Koml\'os, and Szemer\'edi's bound $r(k,n)$ would imply the sharpness of \cref{thm:main}. Equivalently, any asymptotic improvement to \cref{thm:main} for any $N\geq 0$ would improve the upper bound for $r(k,n)$. 

Additionally, it is straightforward using our machinery to show that within the tighter regime of $N=\Theta(n^{k-1}/\log^{k-2}n)$, any asymptotic improvement to \cref{thm:main} would propagate forward via \cref{lem:logboost} to improve the bounds on $r(K,n)$ for all $K>k$. This demonstrates the significant barrier to making any improvements.

\subsection{Warm-up: the basic argument}\label{sec:introproof}

Our proofs only rely on the counts of edges and triangles in graphs with small independence number. We then bootstrap this inductively by averaging over ways to extend an $i$-clique into an $i+1$-clique. In this section, we provide a new proof of \cref{thm:nologs} to demonstrate the basic methods used for the main result. 

\begin{definition}For a graph $G$, let $m_k(G)$ be the number of labeled copies of $K_k$ in $G$; we simply write $m_k$ when $G$ is implicit. Let $S_k$ be the set of such tuples $\vec v=(v_1,\ldots, v_k)$ forming a $k$-clique. Our theorems concern the quantity $m_k/k!$. For $\vec v\in S_k$, let $N(\vec v)$ denote the common neighborhood of $v_1,\ldots, v_k$. For the induction, we have $S_0=\{\emptyset\}$ and $N(\emptyset)=V(G)$.
\end{definition}

Our proof centers around an inductive lemma. We define 
\[
F(x,n) \coloneqq {(n-1)}\binom{x/(n-1)}{2}
=\frac{x^2}{2(n-1)}\Big(1-\frac{n-1}{x}\Big)
\]
By Tur\'an's theorem, $F(N,n)\leq f(N,2,n)$, i.e. every $N$-vertex graph with no independent set of size $n$ contains at least $F(N,n)$ edges. When $N\geq n(n-1)$ we have $F(N,n)\geq N^2/2n$. In the following we choose to work with $F$ because it is convex in $N$. The following lemma relates the clique counts in $G$.

\begin{lemma}\label{lem:nologs} Let $G$ be an $N$-vertex graph with no independent set of size $n$. If $m_{k-1}>0$, then
\[
m_{k+1}\geq m_{k-1} \cdot 2F\Big(\frac{m_k}{m_{k-1}},n\Big)
\]
\end{lemma}
Evaluating $F$ directly yields the following useful versions of the inequality, under different assumptions on the clique ratios.
\begin{corollary}\label{cor:nologs}
In particular, if $m_k/m_{k-1}\geq n^2$ then
\begin{align}\label{eq:nologs}
\frac{m_{k+1}}{m_k} &\geq \frac{1}{n} \frac{m_k}{m_{k-1}},\\
\shortintertext{and if instead \text{$m_k/m_{k-1}\geq Cn$} then
}
\frac{m_{k+1}}{m_k} &\geq \frac{1}{n} \frac{m_k}{m_{k-1}}\Big(1-\frac{1}{C}\Big).\notag
\end{align}
\end{corollary}

The idea of this lemma is that if we build cliques one vertex at a time, each successive neighborhood has density $\sim 1/n$ by Tur\'an's theorem, so the number of choices does not shrink by too much.

\begin{proof}
We enumerate $S_{k-1}$ by choosing $v_1,\ldots, v_{k-1}$, then counting the number of choices for $v_k$ and $v_{k+1}$. This gives
\[
m_{k+1}=\sum_{\vec v\in S_{k-1}}2e(N(\vec v))\geq \sum_{\vec v\in S_{k-1}}2F(|N(\vec v)|,n).
\]
Observe that by counting the choices for only $v_k$, we have $m_k=\sum_{\vec v\in S_{k-1}}|N(\vec v)|$, and thus by convexity of $F$,
\[
m_{k+1}\geq |S_{k-1}|\cdot 2F\bigg(\frac{\sum_{\vec v}|N(\vec v)|}{|S_{k-1}|},n\bigg)= m_{k-1}\cdot 2F\Big(\frac{m_k}{m_{k-1}},n\Big),
\]
as desired.
\end{proof}

This lemma reduces \cref{thm:nologs} to a short calculation. Since we control the ratios, we control the clique counts.

\begin{proof}[Proof of \cref{thm:nologs}] Suppose $N \geq Cn^{k-1}$ for $C>1$. Observe that $m_1/m_0=N$, and then by applying \cref{cor:nologs} inductively, for all $i\leq k-2$ we have $m_{i+1}/m_{i}\geq N/n^i\geq Cn^{k-1-i}$. Finally, since $m_{k-1}/m_{k-2}\geq Cn$, we have $m_k/m_{k-1}\geq (1-\frac 1C)N/n^{k-1}$. In conclusion
\begin{equation}
m_k=\prod_{i=0}^{k-1}\frac{m_{i+1}}{m_i} \geq \Big(1-\frac 1C\Big)\prod_{i=0}^{k-1}Nn^{-i} = \Big(1-\frac 1C\Big)N^kn^{-{\binom{k}{2}}}.\label{eq:nologfinal}
\end{equation}
Set $C=2$ and use $N^k/k!\geq \binom Nk$ to obtain the result.
\end{proof}

\section{Proofs of the main results}\label{sec:mainproof}

The main proofs follow much the same approach as Section \ref{sec:introproof}. The idea is that when we are ``building'' these cliques, we should be able to get a log improvement on the density of each successive neighborhood. Indeed, if the density at any ``level'' is less than $\asymp \frac{\log n}{n}$, then $G$ must be very structured and hence contains many large cliques. In our proofs, we omit floors and ceilings when they are not instructive. By choosing small constants, we may assume $n$ is sufficiently large. Also, in order to simplify notation we avoid choosing explicit constants, for example choosing $c,c',\eps>0$ ``sufficiently small".

We begin with the following theorem which can be obtained from the results of Ajtai, Koml\'os, and Szemer\'edi \cite{AKS}. This is also used in \cite[Theorem 3]{Bohman_Mubayi_2019} and \cite[Lemma 12.16]{Bollobas2001}.

\begin{theorem}[Ajtai--Koml\'os--Szemer\'edi]\label{thm:AKS}
There exists an absolute constant $C$ such that the following holds. If $G$ is an $N$-vertex graph with average degree $d$ and at most $t$ triangles, $t\geq N$, then
\[
\alpha(G) \geq \frac{CN}{d}\log\Big(\frac{d^2N}t\Big)
\]
\end{theorem}

As a corollary, for any graph $G$ with small independent sets, either $G$ has its density boosted by a log factor, or else $G$ contains many more triangles than we require. We remark that the following is exactly the $k=3$ case of \cref{thm:extralog}, up to the choice of $\eps$.

\begin{corollary}\label{cor:AKS} 
For all $\eps>0$, there exists $c>0$ such that the following holds.  If $G$ is an $N$-vertex graph with $N\geq n^{1+\eps}$ and no independent set of size $n$, then either 
\[m_2\geq \frac{c\log n}{n}N^2\quad \text{ or }\quad m_3\geq \frac{N^3}{n^{2+\eps}}.\]
\end{corollary}
\begin{proof}
Let $t=N^3/n^{2+\eps}> N$. Suppose $m_2=c_1\frac{N^2\log n}{n}$ and $m_3\leq t$. We have $d=\frac{m_2}{N}$ and $\frac{N}{d}=\frac{1}{c_1}\frac{n}{\log n}$, so then
\[
n>\frac{C}{c_1}\frac{n}{\log n}\Big(\log\Big(\frac{d^2N}{t}\Big)\Big) \geq \frac{C}{c_1}\frac{n}{\log n}\log(c_1^2n^{\eps}\log^2n) = n\cdot \frac{C}{c_1}\Big(\eps+\frac{2\log(c_1\log n)}{\log n}\Big)
\]
Thus, $c_1$ is bounded below by some constant $c$ in terms of $C,\eps$, as desired.
\end{proof}

Intuitively, the $m_3$ condition is much more efficient for building cliques, and in the $m_2$ case, $G$ is slightly denser by a factor of $\log n$. Bootstrapping this, we obtain that in the desired regime, we either have a log improvement of \cref{lem:nologs}, or else an $n^{1-\eps}$ density boost one step ahead.

\begin{lemma}\label{lem:logboost}
For all $\eps>0$, there exists $c>0$ such that the following holds for all $k\geq 1$. Let $G$ be an $N$-vertex graph with no independent set of size $n$. If $\frac{m_k}{m_{k-1}}\geq n^{1+\eps}$, then either
\[
\frac{m_{k+1}}{m_k}\geq c\frac{\log n}{n}\,\frac{m_k}{m_{k-1}}\quad\text{ or }\quad \frac{m_{k+2}}{m_{k}}\geq \frac{1}{n^{2+\eps}}\,\Big(\frac{m_k}{m_{k-1}}\Big)^2
\]
\end{lemma}

\begin{proof}By choosing $c$ small, we may assume $n$ is sufficiently large. Let $S=S_{k-1}$; we cover $S$ with sets $S^{\rm{small}},S^{(2)},S^{(3)}$, where for $\vec v\in S$, we put 
\[
\vec v\in\begin{cases}
S^{\rm{small}}& \text{if }|N(\vec v)|<\frac{m_k}{2m_{k-1}}\\
S^{(2)}&\text{if }m_2(N(\vec v))\geq \frac{c\log n}
{n}|N(\vec v)|^2\\
S^{(3)}&\text{if }m_3(N(\vec v))\geq \frac{1}{n^{2+\eps}}|N(\vec v)|^3\\
\end{cases}\]
Since $\frac{m_k}{2m_{k-1}}\geq n^{1+\eps/2}$, \cref{cor:AKS} implies that for some choice of $c$, if $\vec v\notin S^{\rm{small}}$ then $\vec v\in S^{(2)}\cup S^{(3)}$, so these three sets cover $S$. Define
\[
m_k^{\rm{small}}=\sum_{\vec v\in S^{\rm{small}}}|N(\vec v)|\leq \sum_{\vec v\in S^{\rm{small}}}\frac{m_k}{2m_{k-1}}\leq m_k/2.
\]
Define $m_k^{(2)},m_k^{(3)}$ similarly, so $m_k^{\rm{small}}+m_k^{(2)}+m_k^{(3)}\geq m_k$. Suppose that $m_k^{(2)}\geq m_k/4$. Then by convexity,
\[
m_{k+1}\geq \sum_{\vec v\in S^{(2)}}m_2(N(\vec v)) \geq |S^{(2)}| \cdot \frac{c\log n}{n}\Big(\frac{m_k/4}{|S^{(2)}|}\Big)^2 \geq c'\frac{\log n}{n} \frac{m_k^2}{m_{k-1}}.
\]
Otherwise, we have $m_k^{(3)}\geq m_k/4$ so
\[
m_{k+2}\geq \sum_{\vec v\in S^{(3)}}m_3(N(\vec v)) \geq |S^{(3)}| \cdot \frac{1}{n^{2+\eps}}\Big(\frac{m_k/4}{|S^{(3)}|}\Big)^3 \geq c'\frac{1}{n^{2+\eps}} \frac{m_k^3}{m_{k-1}^2},
\]
as desired, up to the choice of $c',\eps$.
\end{proof}

We now prove our main theorems by a calculation. The main idea is to show that at each step, we gain a log factor. If not, then we gain many larger cliques more efficiently. Since this second alternative is more powerful, the worst case calculation looks analogous to Equation \eqref{eq:nologfinal}, roughly  $\prod_iN\big(\frac{\log n}{n}\big)^i=N^k(\frac{\log n}{n})^{\binom k2}$. In the last step, we cannot help but lose a log factor. We prove a slightly technical inductive statement which implies both Theorems \ref{thm:main} and \ref{thm:extralog}.
\begin{theorem}\label{thm:maininductive}
For all $\eps>0$ and $k\geq 1$ there exists $c>0$ such that the following holds for all sufficiently large $n$. Let $G$ be an $N$-vertex graph with no independent set of size $n$ and $N\geq n^{k-2+\eps}$. Then there is some critical $1\leq j\leq k-1$ such that for all $0\leq i< j$, we have 
\begin{align*}
\frac{m_{i+1}}{m_i}&\geq N\Big(c\frac{\log n}{n}\Big)^i.\tag{$A_i$}\label{eq:Ai}\\
\shortintertext{If {$j\neq k-1$} then we have}
\frac{m_{j+2}}{m_{j}}&\geq N^2/n^{2j+\eps}\tag{$B_j$}\label{eq:Bi},\\
\shortintertext{and for all {$j\leq i\leq k-2$} we have}
\frac{m_{i+2}}{m_{i+1}}&\geq \alpha_i N/n^{i+\eps},
\tag{$C_i$}\label{eq:Ci}
\end{align*}
where $\alpha_i=2$ for all $i<k-2$ and $\alpha_{k-2}=1$.
\end{theorem}
\begin{proof}Clearly $(A_0)$ holds with $m_1/m_0=N$. If $(A_i)$ holds for all $i\leq k-2$ then we set $j=k-1$ and we are done; otherwise let $j\in [1,k-2]$ be the least index where $(A_j)$ fails. Then by $(A_{j-1})$, 
\[
\frac{m_{j}}{m_{j-1}}\geq N\Big(c\frac{\log n}{n}\Big)^{j-1}\geq n^{k-2+\eps} \Big(c\frac{\log n}{n}\Big)^{k-3} 
>n^{1+\eps},
\]
so by applying \cref{lem:logboost} with $\eps/2$, we have either $(A_j)$ or else 
\[
\frac{m_{j+2}}{m_j}\geq N^2\frac{1}{n^{2+\eps/2}}\Big(c\frac{\log n}{n}\Big)^{2j-2} \gg N^2/n^{2j+\eps},
\]
giving $(B_j)$. Assuming $(A_j)$ fails, we leverage the upper bound on $m_{j+1}/m_j$ to obtain
\[
\frac{m_{j+2}}{m_{j+1}} = \frac{m_{j+2}}{m_j}\frac{m_j}{m_{j+1}}\geq \frac{N^2}{n^{2+\eps/2}}\ \frac{1}{N}\Big(c\frac{\log n}{n}\Big)^{2j-2-j}\gg 2\frac{N}{n^{j+\eps}},
\]
giving $(C_j)$. Now, for induction suppose that $(C_{i-1})$ holds for some $i\in [j+1,k-2]$, so $\frac{m_{i+1}}{m_i}\geq 2N/n^{i-1+\eps}\geq 2n^{k-1-i}$, and we split into cases. If $i<k-2$ then $\frac{m_{i+1}}{m_i}\geq 2n^2$ so by \cref{cor:nologs} we have $\frac{m_{i+2}}{m_{i+1}}\geq 2N/n^{i+\eps}$, giving $(C_i)$. If $i=k-2$ then $\frac{m_{i+1}}{m_i}\geq 2n$ so by \cref{cor:nologs} we have $\frac{m_{i+2}}{m_{i+1}}\geq \frac 12\cdot 2N/n^{i+\eps}$, also giving $(C_i)$. This completes the proof by induction.
\end{proof}

From this, we immediately conclude our two main theorems by simple calculations.
\begin{proof}[Proof of \cref{thm:extralog}]
Let $G$ be an $N$-vertex graph with no independent set of size $n$ and $N\geq n^{k-2+\eps}$, and assume $n$ is sufficiently large. Apply \cref{thm:maininductive} with constant $\eps/2$ to obtain some $j\leq k-1$. If $j<k-1$ then by combining $(A_i),(B_i), (C_i)$ we have
\begin{align*}
m_k&\geq \prod_{i=0}^{j-1}N\Big(c\frac{\log n}{n}\Big)^i\cdot N^2\frac{1}{n^{2j+\eps/2}} \cdot \prod_{i=j+1}^{k-2}N \frac{1}{n^{i+\eps/2}}\\
&\gg N^k n^{-\binom k2} \cdot n^{(k-1-j)(1-\eps/2)}\\
&\gg n^{1-\eps}g(N,k,n).
\end{align*}
Otherwise $j=k-1$, so then for all $\ell\leq k-1$ we have
\[
m_\ell\geq \prod_{i=0}^{\ell -1}N\Big(c\frac{\log n}{n}\Big)^i =\Omega(g(N,\ell,n)).\qedhere
\]
\end{proof}

The other theorem is similar.

\begin{proof}[Proof of \cref{thm:main}]
Let $G$ be an $N$-vertex graph with no independent set of size $n$ and $N\geq c_1n^{k-1}/\log^{k-2}n$, with $c_1$ and $n$ sufficiently large. Apply \cref{thm:maininductive} with constant $\eps=1/2$ to obtain some $j\leq k-1$. If $j<k-1$ then similarly to above, we have $m_k\gg g(N,k,n)$, implying the result. Otherwise $j=k-1$ so similarly we have $m_{k-1}=\Omega(g(N,k-1,n))$. In particular, $(A_{k-2})$ holds so $\frac{m_{k-1}}{m_{k-2}}\geq N\big(c\frac{\log n}{n}\big)^{k-2}>2n$. Then by \cref{cor:nologs}, we have 
\[
\frac{m_{k}}{m_{k-1}}\geq \frac{1}{2n} \frac{m_{k-1}}{m_{k-2}}\geq N c'\frac{\log^{k-2}n}{n^{k-1}}.
\]
We conclude with the desired bound,
\[
m_k= \frac{m_{k}}{m_{k-1}} m_{k-1} \geq \Omega\bigg(N^k \Big(\frac{\log n}{n}\Big)^{\binom{k}{2}}/\log n\bigg).\qedhere
\]
\end{proof}
Observe that by this proof, for any extremal construction $G$ matching this bound, a typical neighborhood of an $i$-clique for $i \leq k-2$ looks random-like, whereas a typical neighborhood of a $(k-1)$-clique is Tur\'an-like, having lower density but more structure (e.g., complete multipartite). In \cref{lem:blowups}, we use blowups to achieve this behavior.

\subsection{Upper bounds}\label{sec:upperbounds}

First, we show that once we have one (slightly stronger) upper bound witness, we can use blow-ups to obtain upper bounds for all $N$.

\begin{lemma}\label{lem:blowups}Suppose for some $x>0$, $R\geq 1$, there exists an $R$-vertex graph $G$ with no independent set of size $n$ and at most $R^kx$ cliques of size \emph{at most} $k$. Then for all $N\geq 1$, letting $C=k^k$, we have
\[
f(N,k,n)\leq C\, N^k x.
\]
\end{lemma}
\begin{proof}First, suppose $N$ is a positive multiple of $R$. Obtain $G_+$ from $G$ by replacing each vertex with a clique of size $N/R$. Each $k$-clique in $G_+$ is contained in a blow-up of an $i$-clique of $G$ for some $i\leq k$. Counting labeled cliques, we have
\[
m_k(G_+)\leq \sum_{i\leq k}m_i(G)(iN/R)^k \leq k^k\, k!\, N^k x.
\]
Divide by $k!$ to obtain the result for $N$. Now, suppose $N'<N$, and we prove the result for $N'$. Let $G'\subseteq G_+$ be an induced subgraph on a random subset of $N'$ vertices. Then $G'$ contains no independent set of size $n$ and $\Ex[m_k(G')]\leq (N'/N)^k m_k(G^+)$, so there is some choice of $G'$ giving the desired bound.
\end{proof}

Taking $G\sim \bG(N,p)$ with $N=n^k$ and $p=2k(\log n)/n$, we have with high probability that $\alpha(G)<n$  and $m_{\leq k}(G) \sim N^kp^{-\binom k2}$, so we conclude our general upper bound for $f$. This shows the bound of \cref{thm:main} is tight up to a factor of $\log n$.

\begin{corollary}\label{cor:randomupperbound} For all $N,n,k\geq 2$, we have $f(N,k,n)\leq O_k(g(N,k,n))$.
\end{corollary}

Now, we prove that \cref{thm:main} is conditionally tight for all $k$, conditioned on the asymptotics of $r(k,n)$ being tight to the upper bound of \eqref{eq:ramseybounds}. In order to apply \cref{lem:blowups} to prove \cref{thm:tightness}, it suffices to show that a particular Ramsey graph on $R$ vertices would have $O(g(R,k,n)/\log n)$ cliques of size at most $k$. Towards this, define $m_{\leq k}=\sum_{i\leq k}m_i$.

\begin{proof}[Proof of \cref{thm:tightness}] Suppose that $R=r(k,n)-1=\Theta(n^{k-1}/\log^{k-2}n)$. Let $G$ be an $R$-vertex graph with no independent set of size $n$ and no clique of size $k$. For $i\leq k-1$, we count $i$-cliques by choosing vertices $v_1,\ldots, v_i$ sequentially. For $j\leq i$, $N(v_1,\ldots, v_j)$ is $K_{k-j}$-free so we must have $|N(v_1,\ldots, v_{j})|<r(k-j,n)$. This gives
\begin{align*}
m_i(G)< \prod_{j=0}^{i-1}r(k-j,n) \leq O\bigg(\prod_{j=0}^{i-1}n^{k-j-1}/\log^{k-j-2}n\bigg)&= O\big(R^i(\log n/n)^{\binom i2}\big).
\end{align*}
Notably, this is necessarily tight for all $i$ by \cref{thm:extralog}. This gives the bound $m_{\leq k}(G)\leq O(R^{k-1}(\log n/n)^{\binom{k-1}2})=O(g(R,k,n)/\log n)$ by choice of $R$. Then by \cref{lem:blowups} we conclude $f(N,k,n)\leq O(g(N,k,n)/\log n)$ for all $N$, as desired.
\end{proof}

We now show that \cref{thm:main} is tight over its domain when $k=4$. The proof ideas for the following theorem were suggested by ChatGPT 5.6 Sol after reviewing an early draft of the paper. The proof is written by the author.

\begin{theorem}\label{thm:k=4tight}There exist $C_0,C_1>0$ such that for all $n\ge 2$ and
$N\ge C_0n^3/(\log n)^2$, there is an
$N$-vertex graph $G$ with $\alpha(G)<n$ and at most
\[
C_1\binom{N}{4}
\left(\frac{\log n}{n}\right)^6/\log n
\]
cliques of size $4$. Together with \cref{thm:main}, we have
\[
f(N,4,n)=\Theta\left(g(N,4,n)/\log n\right)
\]
throughout this range.
\end{theorem}

Our graph is built from a $K_4$-free construction  of Mattheus and Verstraete. We blow up each vertex into a clique and randomly sparsify the joins between adjacent parts. 

Interestingly, their lower bound on $r(4,n)$ leaves a $\log^2n$ gap from the upper bound, yet their intermediate construction gives an optimal bound for $f$ up to constants. The following is implied by their construction which they denote $H_q^*$ for a prime power $q=\Theta(\sqrt n)$; see Theorem 3 and Proposition 2 in \cite{r4k}.

\begin{theorem}[Mattheus--Verstraete]\label{thm:k4freebase} There exists $n_0\geq 1$ such that for $n\geq n_0$, there is a $K_4$-free graph $H$ with $\Theta(n^2)$ vertices and $\Theta(n^{3.5})$ edges such that every subset of $n$ vertices in $H$ contains at least $10n^{3/2}$ edges.
\end{theorem}

\begin{proof}[Proof of \cref{thm:k=4tight}] Up to the choice of $C_1$, we may assume $n\geq n_0$. Take $H$ as in \cref{thm:k4freebase}, and we shall define $G$ as a randomly sparsified blow-up of $H$. Specifically, obtain $H_+$ by replacing each vertex $v\in V(H)$ with a clique $U_v$ of size $\lambda=\lfloor n/\log^2 n\rfloor$, and let $R=|V(H_+)|=\lambda |V(H)|$. Then, obtain a spanning subgraph $G$ from $H_+$ by keeping each cross edge $xy$ for $x\in U_u, y\in U_v,uv\in E(H)$ with probability $p=n^{-1/2}\log n$. Do not delete any edge within a part $U_v$. First, we check that $G$ has no independent set of size $n$. For every subset $S\subseteq V(G)$ of size $n$, $S$ can only be independent if each vertex lies in a different part $U_v$, in which case
\[
\PP[S\text{ is independent}]= (1-p)^{e_{H_+}(S)}\leq e^{-10pn^{3/2}}=n^{-10n}\ll \binom Rn^{-1},
\]
where we use $p=\frac{\log n}{n^{1/2}}$ and $R=\Theta(n^3/\log^2 n)$. Taking a union bound over all sets $S$, we have $\alpha(G)<n$ with high probability. Next, we estimate the clique counts of $G$. Using $m_{\leq k}(G)=\sum_{i\leq k}m_i(G)$, we will show:
\begin{equation}\label{eq:k4target}
\Ex[m_{\leq 4}(G)] \leq O\Big(R^4\frac{\log^5n}{n^6}\Big) = O\Big(\frac{n^6}{\log^3n}\Big).
\end{equation}
First, recall that $m_4(H)=0$, so the contrapositive of \cref{cor:nologs} gives $\frac{m_3(H)}{m_2(H)}\leq n$, and thus $m_3(H)\leq O(n^{4.5})$. To bound $m_i(G)$, we count for each $j\leq i$ the number of ways to find an $i$-clique of $G$ spanning a blown-up $j$-clique of $H$, where the probability depends on the number of edges crossing the $j$ corresponding parts of $H_+$.
Denote this count by $m_{i,j}(G)$, and we observe that the fewest crossing edges in a surjective map $K_i\to K_j$ is $\binom i2-\binom{i-j+1}2=i(j-1)-\binom j2$. Thus, using $\lambda=p^{-2}$ we have
\[
\Ex[m_{i,j}(G)]\leq O\Big(m_j(H)\lambda^i p^{i(j-1)-\binom j2 }\Big) = O\Big(m_j(H)p^{i(j-3)-\binom j2}\Big).
\]
For all $j\leq 3$, our bound on $\Ex m_{i,j}$ is increasing in $i$. It remains to consider $i=4$; we have
\[
\Ex[m_4(G)]\leq O\Big(n^2p^{-8}+n^{3.5}p^{-5}+n^{4.5}p^{-3}\Big)=O\Big(\frac{n^6}{\log^3n}\Big).
\]
This proves Equation \eqref{eq:k4target}. By Markov's inequality, we can select a particular graph $G$ on $R$ vertices with $\alpha(G)<n$ and $m_{\leq 4}(G) = O(g(R,4,n)/\log n)$. Then by \cref{lem:blowups}, we have $f(N,4,n)\leq O(g(N,4,n)/\log n)$ for all $N$.
\end{proof}

\section{Concluding Remarks}\label{sec:conclusion}

The methods in this paper can be used to obtain stronger bounds for graphs with additional properties. Our technique specifically tracks the \textit{ratios} $m_{k+1}/m_k$, so any further bounds on clique counts will yield a different result. For example, we have the following.
\begin{proposition}If $G$ satisfies the conditions of \cref{thm:main} and further $G$ has average degree $d=\frac{m_2}{m_1}$ then
\[
m_k(G)\geq \Omega\bigg(Nd^{k-1}\Big(\frac{\log n}{n}\Big)^{\binom{k-1}2}/\log n\bigg).
\]
\end{proposition}
This can be obtained by similarly to \cref{thm:main} by applying \cref{lem:logboost} inductively.

When $k\geq 5$, our bounds leave a gap on the order of $\log n$. The recent determination of $r(k,n)=\tilde\Theta(n^{k-1})$ brings renewed attention to the remaining logarithmic factors, which is further motivated by Theorems \ref{thm:tightness} and \ref{thm:k=4tight}. The main problem we leave open is the final logarithmic factor.

\begin{problem}For $k\geq 5$ and $N$ sufficiently large in terms of $n$, do we have $f(N,k,n)=\Theta(g(N,k,n)/\log n)$?
\end{problem}

Disproving this would improve the upper bound on $r(k,n)$. Any improvement on the upper bound would likely use a quasirandom graph which is `more effective' than $\bG(N,p)$ at avoiding $K_k$. Constructions of this general type already exist in the intermediate regime $n^{(k-1)/2}\ll N\ll n^{k-1}$, as witnessed by $r(k,n)$ (also see \cref{thm:k4freebase}).

For a fixed graph $H$, we define $f(N,H,n)$ as the minimum number of unlabeled copies of $H$ in an $N$-vertex graph with no independent set of size $n$, so $f(N,K_k,n)=f(N,k,n)$. We have the following theorem generalizing \cref{thm:nologs}. For complete multipartite graphs $H$, we do not know the polynomial order of $r(H,n)$, but we can solve the multiplicity problem up to log factors.
\begin{theorem}\label{thm:multipartite}
If $H$ is a complete multipartite graph and $\eps>0$, there is a constant $C$ such that if $N\geq Cn^{\delta(H)}$, where $\delta(H)$ is the minimum degree of $H$, then
\[
f(N,H,n)\geq \frac{(1-\eps)}{|\rm{Aut}(H)|}
N^{|V(H)|} n^{-|E(H)|}.
\]
\end{theorem}
This is tight up to polylogarithmic factors (and the bound on $N$), where the matching upper bound is once again from $\bG(N,p)$. The proof is deferred to Appendix \ref{sec:multipartite}; it is analogous to the proof of \cref{thm:nologs}, where instead of extending $K_k$ into $K_{k+1}$ by finding edges $K_{1,1}$ inside some neighborhood $N(K_{k-1})$, we perform the same extensions by finding $K_{s,t}$ inside some $N(H')$. We pose the following general problem, which may be more tractable for specific classes of graphs.

\begin{problem}For fixed $H$ and $N$ sufficiently large in terms of $n$, determine $f(N,H,n)$. For what graphs $H$ is there a phase transition at $N\approx r(H,n)$?
\end{problem}

\section*{Acknowledgments}

I would like to thank my advisor Xiaoyu He for his guidance, comments, and helpful discussions, particularly on a related problem in the hypergraph setting, which ultimately inspired this project. Following an early draft of this paper, the proof ideas for \cref{thm:k=4tight} were suggested by ChatGPT 5.6 Sol; the author independently verified and wrote the proof. All other proofs and ideas were derived without significant use of AI tools.

\newpage

\input{Compilation_Files/biblio}

\newpage
\appendix
\setlength{\abovedisplayskip}{6.5pt}
\setlength{\belowdisplayskip}{6.5pt}
\section{Complete multipartite graphs}\label{sec:multipartite}

For $t_1,\ldots, t_k\geq 1$ and a graph $G$, let $m_{t_1,\ldots, t_k}(G)$ be the number of labeled copies of the complete multipartite graph $K_{t_1,\ldots,t_k}$ in $G$. 

We begin by proving a quantitative version of our theorem when $H$ is bipartite. In the analogous argument for \cref{thm:nologs}, this corresponds to the choice of function $F(x,n)$, which works as intended by Tur\'an's theorem. We then use this density result to bootstrap from smaller multipartite graphs to larger ones. It would be interesting to obtain a polylogarithmic improvement on the following.

\begin{proposition}\label{lem:bipartite} If $s,t,n,N\geq 1$ and $G$ is an $N$-vertex graph with no independent set of size $n+1$, then 
\[m_{s,t}(G)\geq N^{s+t}n^{-st}\Big(1-s\binom {t+1}2\frac{n}{N}-\binom s2\frac{n^t}{N}\Big).
\]
In particular, if $s,t$ are fixed and $n^t=o(N)$ then the error is $o(1)$.
\end{proposition}
The asymptotic version of the proposition above for $N$ sufficiently large is almost immediate: use Tur\'an's theorem to bound $m_{1,1}$, and then the Sidorenko property to bound $m_{s,t}$. The exact error term is important in our regime where $n$ grows with $N$. Denote the falling factorial by $x^{\underline t}=\prod_{i=0}^{t-1}(x-i)$, and define the function
\[
h(x,t)\coloneqq \max\bigg(0,x^t\Big(1-\binom t2\frac 1x\Big)\bigg),
\]
We use that $h(x,t)$ is convex in $x$ and for all non-negative integers $N$, $h(N,t)\leq N^{\underline t}\leq N^t$.

\begin{proof}We may assume $N\geq n$ or else the bound is non-positive. First, we check when $s=1$. The average degree of $G$ is at least $\frac{N}{n}-1$ by Tur\'an's theorem. Then by convexity,
\begin{align*}
m_{1,t}&=\sum_{v\in V(G)}|N(v)|^{\underline t}\geq \sum_{v\in V(G)}h(|N(v)|,t)\geq N\,h\Big(\frac Nn-1,t\Big) \geq N^{t+1}n^{-t}\Big(1-\binom{t+1}2\frac{n}{N}\Big).
\end{align*}
Let $S_{t,0}$ be the set of labeled $t$-sets of $V(G)$, so $|S_{t,0}|=N^{\underline t}$. Then by convexity again,
\begin{align*}
m_{s,t}=\sum_{\vec v\in S_{t,0}}|N(\vec v)|^{\underline s} 
\geq \sum_{\vec v\in S_{t,0}}h(N(\vec v),s)
\geq N^{\underline t}\,h\Big(\frac{m_{1,t}}{N^{\underline t}},s\Big)\geq N^t\,h\Big(\frac{m_{1,t}}{N^t},s\Big),
\end{align*}
using that $h(x,s)/x$ is nondecreasing for non-negative $x$. Then
\begin{align*}
m_{s,t}
&\geq N^t\,h\Big(Nn^{-t}(1-\tbinom{t+1}2n/N),t\Big)\\
&\geq N^{s+t}n^{-st}\Big(1-s\binom{t+1}2\frac nN-\binom s2 \frac{1}{Nn^{-t}}\Big).\qedhere
\end{align*}
\end{proof}
We define 
\[
F_{s,t,n}(N) = \max(\RHS,0),
\]
where the $\RHS$ is from \cref{lem:bipartite}. 
When $s=t=1$, this specializes to the Tur\'an bound $F(x,n+1)=\frac{x^2}{2n}(1-\frac nx)$ (the factor of $2$ arises from labeling). Let $F^*(x)=F_{s,t,n}(x^{1/t})$, so 
\[
F^*(x)=\max(0,x^{1+s/t}-C_{s,t,n}x^{1+s/t-1/t})n^{-st},
\]
for some $C_{s,t,n}>0$. We now check that $F^*$ is convex and nondecreasing on $\R_{\geq 0}$; indeed we can replace $F^*$ with any function $g(x)=\max(0,x^a-Cx^b)$ with $a\geq 1$ and $a\geq b\geq 0$ and $C>0$. It suffices to show these properties when $x\geq C^{\frac{1}{a-b}}$. Then we have
\begin{align*}
g'(x)&=ax^{a-1}-bCx^{b-1}\geq 0\\
g''(x)&=a(a-1)x^{a-2}-b(b-1)Cx^{b-2}\geq 0
\end{align*}
since $a\geq b$, $a(a-1)\geq b(b-1)$ and $Cx^{b-a}\leq 1$. In the following lemma, we generalize \cref{lem:nologs} by considering the average number of extensions from smaller subgraphs to larger ones.

\begin{lemma}Let $G$ be an $N$-vertex graph with no independent set of size $n$. For $t_1,\ldots, t_{k-1},s,t\geq 1$, if we have $m_{t_1,\ldots,t_{k-1}}>0$, then
\[
m_{t_1,\ldots, t_{k-1},t,s} \geq m_{t_1,\ldots,t_{k-1}}F^*\Big(\frac{m_{t_1,\ldots,t_{k-1},t}}{m_{t_1,\ldots,t_{k-1}}}\Big)
\]
\end{lemma}
\begin{proof} Let $S=S_{t_1,\ldots, t_{k-1}}$ be the set of labeled copies of $K_{t_1,\ldots, t_{k-1}}$, and note that $m_{t_1,\ldots, t_{k-1},t}=\sum_{\vec v\in S}|N(\vec v)|^{\underline t}$. Then using that $F^*$ is convex and nondecreasing,
\begin{align*}
m_{t_1,\ldots, t_{k-1},s,t} = \sum_{\vec v\in S}m_{s,t}(N(\vec v)) 
\geq  \sum_{\vec v\in S}F_{s,t,n}(|N(\vec v)|)
&\geq \sum_{\vec v\in S}F^*(|N(\vec v)|^{\underline t})\\
&\geq m_{t_1,\ldots,t_{k-1}}F^*\Big(\frac{m_{t_1,\ldots,t_{k-1},t}}{m_{t_1,\ldots,t_{k-1}}}\Big).\qedhere
\end{align*}
\end{proof}
Using the estimate $F^*(x)\approx x^{1+s/t}n^{-st}$ for $x\gg n^{t^2}$, we obtain the following corollary.
\begin{corollary}\label{cor:bipextensions} For fixed $s,t\geq1$, $\eps>0$, there exists $C$ such that if $n\geq 1$ and $\frac{m_{t_1,\ldots,t_{k-1},t}}{m_{t_1,\ldots,t_{k-1}}}\geq Cn^{t^2}$, then
\[
\frac{m_{t_1,\ldots, t_{k-1},t,s}}{m_{t_1,\ldots, t_{k-1},t}}\geq (1-\eps)\Big(\frac{m_{t_1,\ldots,t_{k-1},t}}{m_{t_1,\ldots,t_{k-1}}}\Big)^{\frac{s}{t}}n^{-st}
\]
\end{corollary}

Using this, we prove \cref{thm:multipartite} by an inductive calculation. For a graph $H$, let $m_H(G)$ be the number of labeled copies of $H$ in $G$. We prove the following version of the theorem.

\begin{theorem}Let $G$ be an $N$-vertex graph with no independent set of size $n$. For fixed $t_1,\ldots, t_{k-1},t\geq 1$ and $\eps>0$ there exists $C>0$ such that if $n\geq 1$ and $N\geq Cn^{t_1+\cdots+t_{k-1}}$ then
\[
\frac{m_{t_1,\ldots,t_{k-1},t}}{m_{t_1,\ldots,t_{k-1}}}\geq (1-\eps)N^tn^{-t(t_1+\cdots+t_{k-1})}
\]
Consequently, letting $H=K_{t_1,\ldots,t_{k-1},t}$, if $N\geq Cn^{\delta(H)}$, we have
\[
m_H(G)\geq (1-\eps)^k N^{|V(H)|} n^{-|E(H)|}.
\]
\end{theorem}
The `consequently' part is immediate by writing $m_H$ as a product of consecutive ratios. Indeed, the ratio in this lemma exactly corresponds to the incremental difference in the vertex and edge counts by adding a new part of size $t$. Note that we can set $t_1+\cdots+t_{k-1}=\delta(H)$ by intentionally choosing the last part to be the largest.
\begin{proof}
The base case follows directly from \cref{lem:bipartite}. For the inductive case, let $s\geq 1$ and assume $N\geq Cn^{t_1+\ldots+ t_{k-1}+t}$, so we have 
\[\frac{m_{t_1,\ldots,t_{k-1},t}}{m_{t_1,\ldots,t_{k-1}}}\geq (1-\eps)\big(Nn^{-(t_1+\cdots+ t_{k-1})}\big)^t\geq Cn^{t^2}.
\]
Now by \cref{cor:bipextensions}, we have
\begin{align*}
\frac{m_{t_1,\ldots, t_{k-1},t,s}}{m_{t_1,\ldots, t_{k-1},t}}&\geq (1-\eps)\Big((1-\eps)N^t n^{-t(t_1+\cdots+t_{k-1})}\Big)^{s/t}n^{-st}\\
&\geq (1-\eps')N^sn^{-s(t_1+\cdots+t_{k-1}+t)},
\end{align*}
as desired, up to the choice of $\eps$.
\end{proof}

\end{document}

%% file: Compilation_Files/preamble.tex
\usepackage{amsmath,amsthm,amsfonts,graphicx,float}
\usepackage{amssymb,upgreek,stmaryrd,framed,bbm}
\usepackage{wrapfig,mathrsfs,xfrac,mathtools,epsdice}
\usepackage{qtree,setspace,accents,verbatim}
\usepackage[x11names]{xcolor}
\usepackage[ruled,vlined]{algorithm2e}
\usepackage[normalem]{ulem}
\usepackage[english]{babel}
\usepackage{framed}
\usepackage[shortlabels]{enumitem}

\usepackage{thmtools,thm-restate}

\usepackage{tikz}
\usetikzlibrary{shapes}

\makeatletter\let\@@citation@@=\citation\renewcommand{\citation}[1]{\@@citation@@{#1}\@for\@tempa:=#1\do{\@ifundefined{cit@\@tempa}{\global\@namedef{cit@\@tempa}{}}{}}}\makeatother
\usepackage[linktocpage=true]{hyperref}
\usepackage[capitalize]{cleveref}
\makeatletter\def\@lbibitem[#1]#2#3\par{\@ifundefined{cit@#2}{}{\@skiphyperreftrue\H@item[\ifx\Hy@raisedlink\@empty\hyper@anchorstart{cite.#2\@extra@b@citeb}\@BIBLABEL{#1}\hyper@anchorend\else\Hy@raisedlink{\hyper@anchorstart{cite.#2\@extra@b@citeb}\hyper@anchorend}\@BIBLABEL{#1}\fi\hfill]\@skiphyperreffalse}\if@filesw\begingroup\let\protect\noexpand\immediate\write\@auxout{\string\bibcite{#2}{#1}}\endgroup\fi\ignorespaces\@ifundefined{cit@#2}{}{#3}} \def\@bibitem#1#2\par{\@ifundefined{cit@#1}{}{\@skiphyperreftrue\H@item\@skiphyperreffalse\Hy@raisedlink{\hyper@anchorstart{cite.#1\@extra@b@citeb}\relax\hyper@anchorend}}\if@filesw\begingroup\let\protect\noexpand\immediate\write\@auxout{\string\bibcite{#1}{\the\value{\@listctr}}}\endgroup\fi\ignorespaces\@ifundefined{cit@#1}{}{#2}}\makeatother

\newcommand{\Ex}{\mathbb{E}}

\newcommand{\PP}{\mathbb{P}}

\newcommand{\R}{\mathbb{R}}

\newcommand{\cP}{\mathcal{P}}

\newcommand{\bG}{\mathbb{G}}

\newcommand{\bR}{\mathbb{R}}

\newcommand{\eps}{\upvarepsilon}

\renewcommand{\rm}[1]{\mathrm{#1}}

\def\multichoose#1#2{\ensuremath{\left(\kern-.3em\left(\genfrac{}{}{0pt}{}{#1}{#2}\right)\kern-.3em\right)}}

\DeclareMathOperator{\poly}{poly}

\DeclareMathOperator{\RHS}{RHS}

\newtheorem{theorem}{Theorem}[section]
\newtheorem{lemma}[theorem]{Lemma}
\newtheorem{problem}[theorem]{Problem}
\newtheorem{proposition}[theorem]{Proposition}
\newtheorem{corollary}{Corollary}[theorem]

\theoremstyle{definition}
\newtheorem*{definition}{Definition}

\theoremstyle{remark}
